\documentclass[11pt]{article}
\usepackage{graphicx}
\usepackage{setspace}
\usepackage{amssymb,amsmath}
\usepackage{calc}
\usepackage{verbatim}
\usepackage{epsfig}
\usepackage{graphicx}
\usepackage{graphics}
\usepackage{cite}
\usepackage[percent]{overpic}
\usepackage{tikz}
\tikzstyle{vertex}=[circle, draw, inner sep=0pt, minimum size=5pt,fill=black]
\newcommand{\vertex}{\node[vertex]}
\usetikzlibrary{decorations.pathreplacing}

\usepackage{amsmath, amsfonts, amssymb,amsthm,fancyhdr}
\usepackage{latexsym,amsmath}

\newtheorem{theorem}{Theorem}[section]
\newtheorem{lemma}[theorem]{Lemma}

\newtheorem{corollary}[theorem]{Corollary}
\newtheorem{problem}[theorem]{Problem}
\newtheorem{proposition}[theorem]{Proposition}

\newcommand{\diam}{D}

\renewcommand{\Re}{\mathrm{Re}}
\renewcommand{\Im}{\mathrm{Im}}

\begin{document}

\title{On the Roots of Wiener Polynomials of Graphs}
\author{JASON I. BROWN\footnote{Supported by an NSERC grant CANADA, Grant number RGPIN 170450-2013} \\
\small{Department of Mathematics and Statistics}\\
\small{Dalhousie University, Halifax, NS, CANADA}\\
\small{jason.brown@dal.ca}\\
LUCAS MOL and ORTRUD R. OELLERMANN\footnote{Supported by an NSERC grant CANADA, Grant number RGPIN 05237-2016}\\
 \small{Department of Mathematics and Statistics}\\
 \small{University of Winnipeg, Winnipeg, MB, CANADA}\\
 \small{l.mol@uwinnipeg.ca, o.oellermann@uwinnipeg.ca}}
\date{}

\maketitle

\begin{abstract}
The {\em Wiener polynomial} of a connected graph $G$ is defined as  $W(G;x)=\sum x^{d(u,v)}$, where $d(u,v)$ denotes the distance between $u$ and $v$, and the sum is taken over all unordered pairs of distinct vertices of $G$. We examine the nature and location of the roots of Wiener polynomials of graphs, and in particular trees. We show that while the maximum modulus among all roots of Wiener polynomials of graphs of order $n$ is $\binom{n}{2}-1$, the maximum modulus among all roots of Wiener polynomials of trees of order $n$ grows linearly in $n$.  We prove that the closure of the collection of real roots of Wiener polynomials of all graphs is precisely $(-\infty, 0]$, while in the case of trees, it contains $(-\infty, -1]$.  Finally, we demonstrate that the imaginary parts and (positive) real parts of roots of Wiener polynomials can be arbitrarily large.\\

\noindent
\textbf{MSC 2010:} 05C31, 05C12

\noindent
\textbf{Keywords:} Graph polynomials; Wiener polynomial; Polynomial roots; Wiener index; Distance in graphs
\end{abstract}

\section{Introduction}

Let $G$ be a connected graph of order at least $2$ and {\em diameter} $D(G)$.  The {\em Wiener polynomial} of $G$ is defined as
\[
W(G;x)=\sum x^{d(u,v)},
\]
where $d(u,v) = d_G(u,v)$ is the {\em distance} between $u$ and $v$, and the sum is taken over all distinct unordered pairs of vertices $u$ and $v$ of $G$. Thus if $d_i(G)$ denotes the number of pairs of vertices distance $i$ apart in $G$, for $i = 1,2,\ldots,D(G)$, then
\[
W(G;x) = \sum_{i=1}^{D(G)} d_i(G)x^i.
\]
If $G$ is clear from context, then we write $d_i$ and $D$ instead of $d_i(G)$ and $D(G),$ respectively.  Note that we restrict our discussions to connected graphs of order at least $2$ throughout, so that the Wiener polynomial is well-defined and nonzero.  Adding loops or multiple edges to a connected graph has no effect on the Wiener polynomial, so we lose no generality by considering only simple graphs.

The successful study of the chromatic~\cite{Chromatic}, characteristic~\cite{Characteristic}, independence~\cite{Independence} and reliability~\cite{Reliability} polynomials of graphs has motivated the study of various other graph polynomials.  The Wiener polynomial was introduced in \cite{Hos} and independently in \cite{SYZ}, and has since been studied several times (see \cite{di,GKYY, WRSSG, YYY}, for example).  Unlike many other graph polynomials (such as the chromatic, reliability and independence polynomials), the Wiener polynomial can be calculated in polynomial time.

Interest in the Wiener polynomial first arose out of work on the {\em Wiener index} of a connected graph, introduced in \cite{Wien} and defined as the sum of the distances between all pairs of vertices of the graph. It can readily be seen that the Wiener index of $G$ is obtained by evaluating the derivative of the Wiener polynomial of $G$ at $x=1$.  The Wiener index of a molecular graph is closely correlated with certain physical properties of the substance, such as its melting and boiling point; see \cite{Wien, GK, ST, RC}.  Applications of the Wiener index have also been found outside of chemisty, for example in the design of architectural floor plans~\cite{architect}.

Wiener polynomials of trees will be of particular interest to us, so we mention that the Wiener polynomial of a tree also arises in the context of network reliability.  Given a graph $G$ in which each edge is operational with probability $p$, the \emph{resilience} or \emph{pair-connected reliability} of $G$ is the expected number of pairs of vertices of $G$ that can communicate~\cite{Colbourn, ASS1993}.  In particular, in a tree $T$, the probability that any pair of vertices $u$ and $v$ can communicate is simply $p^{d(u,v)}$.  Therefore, the resilience of $T$ is exactly $W(T;p)$.

In this article, we focus on the study of the roots of the Wiener polynomial of a graph $G$, which we call the \emph{Wiener roots} of $G$.  The Wiener roots of all connected graphs of order $8$ are shown in Figure~\ref{order8}.  Though the Wiener roots of some narrow families of graphs were studied in \cite{di}, little is known about the nature and location of the Wiener roots of graphs in general.  Given that properties of the roots of other graph polynomials have been well-studied, it is natural to study the roots of the Wiener polynomial in greater depth.

\begin{figure}[ht]
\centering
\begin{tikzpicture}
\node[inner sep=0pt,left] (plot) at (0,0)
    {\includegraphics[scale=0.5]{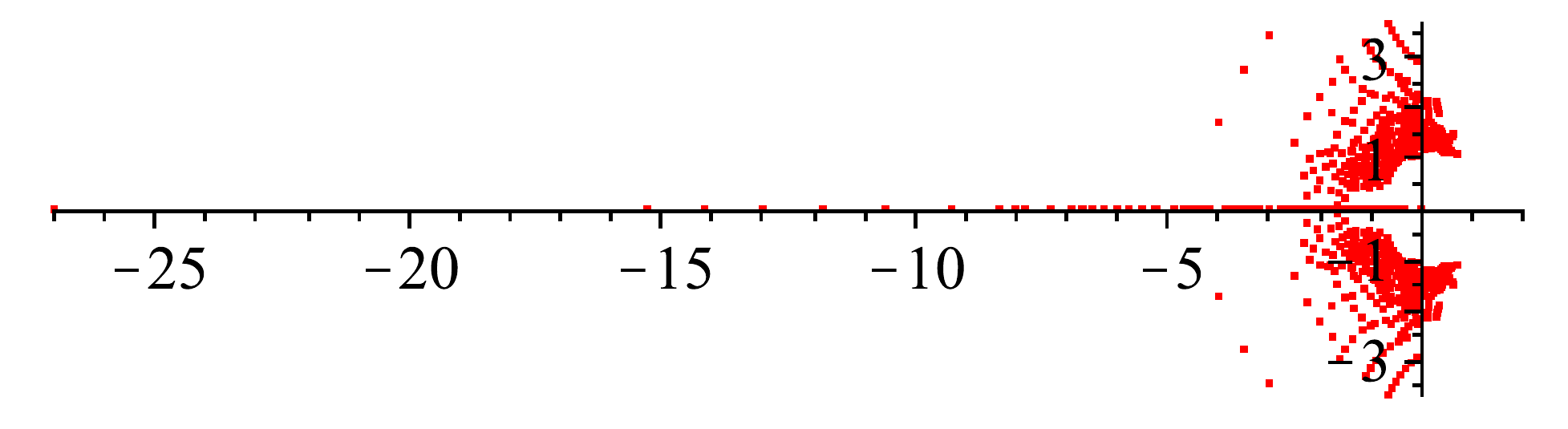}};
\node[right] at (0,-0.05) {\footnotesize $\mbox{Re}(z)$};
\node at (-0.55,1.4) {\footnotesize $\mbox{Im}(z)$};
\end{tikzpicture}
\caption{The Wiener roots of all connected graphs of order $8$.}
\label{order8}
\end{figure}

Since $W(G;x) = x(d_1+d_2x+ \ldots +d_Dx^{D-1})$, we see that $x=0$ is a root of every Wiener polynomial. The remaining roots are those of the polynomial $\widehat{W}(G;x) = d_1+d_2x+ \ldots +d_Dx^{D-1}$, which we will refer to as the {\em reduced Wiener polynomial} of $G$. Note that the value of $d_{1}$ is the number of edges, so that if the order of $G$ is at least $2$, then $0$ is \emph{not} a root of $\widehat{W}(G;x)$ (and hence is a simple root of the original Wiener polynomial $W(G;x)$).

The layout of the article is as follows.  In Section~\ref{BoundSection}, we determine that $\binom{n}{2}-1$ is the maximum modulus among all Wiener roots of connected graphs of order $n$, and we describe the extremal graphs which have a root of this maximum modulus.  In contrast, we then show that for trees, the rate of growth of the maximum modulus is $\Theta(n)$.  Finally, we determine that $\tfrac{2}{n-2}$ is the minimum modulus among all nonzero Wiener roots of connected graphs of order $n$, and again we are able to describe the extremal graphs (which are trees).  In Section~\ref{DensitySection}, we prove that the closure of the collection of all real Wiener roots of connected graphs is $( -\infty, 0]$, and that for trees, the closure of all real Wiener roots still contains $(-\infty,-1]$.  In Section~\ref{RealAndImagSection}, we study the real and imaginary parts of Wiener roots.  We show that there are Wiener roots of trees with arbitrarily large imaginary part, and that there are Wiener roots of trees with arbitrarily large positive real part.  This work culminates in a proof of the fact that the collection of complex Wiener roots of all graphs is not contained in any {\em half-plane} (a region in the complex plane consisting of all points on one side of a line and no points on the other side).  Finally, we prove that almost all graphs have all real Wiener roots, and we find purely imaginary Wiener roots.  Throughout, we compare and contrast our results with what is known about the roots of other graph polynomials.

\section{Bounding the Modulus of Wiener Roots}\label{BoundSection}

Bounding the modulus of roots of various graph polynomials has been a central point of interest in the literature. For example, Sokal~\cite{S} settled an outstanding problem in proving that the modulus of a chromatic root (a root of the chromatic polynomial) of a graph of maximum degree $\Delta$ is at most a constant times $\Delta$.  This implies that the modulus of any chromatic root of a graph of order $n$ is at most $Cn$ for some positive constant $C$; the precise value of the smallest such constant is still unknown.  The \emph{independence polynomial} of a graph $G$ is the generating function for the number of independent sets of each cardinality in $G$.  Among all graphs of order $n$ and fixed independence number $\beta$, the maximum modulus of an independence root is $(n/\beta)^{\beta-1}+O\left({n}^{\beta - 2}\right)$~\cite{BN}.  For the \emph{all-terminal reliability} polynomial, it is suspected that the collection of all roots is bounded in modulus (as no roots have been found outside of the disk $|z-1|\leq 1.2$). Currently the best known upper bound on the modulus of the roots for all graphs of order $n$ is linear in $n$ as shown in ~\cite{BM}.  For the independence polynomials, the all-terminal-reliability polynomials and the chromatic polynomials, the extremal graphs that yield roots of maximum modulus are not known. In contrast we completely determine the extremal graphs for which the Wiener polynomials, of all connected graphs of order $n$, have maximum modulus of any root.  To do so, we will make use of the classical Enestr\"{o}m-Kakeya Theorem~\cite{EK}.

\begin{theorem}[The Enestr\"{o}m-Kakeya Theorem]
If $f(x) = a_0+a_1x+ \ldots+a_nx^n$ has positive real coefficients, then all complex roots of $f$ lie in the annulus
\[r \le |z| \le R\]
where $r =\min \left\{\frac{a_i}{a_{i+1}}: 0 \le i \le n-1\right\}$ and $R =\max\left\{\frac{a_i}{a_{i+1}}: 0 \le i \le n-1\right\}$.\hfill \qed
\end{theorem}

As a straightforward application of the Enestr\"om-Kakeya Theorem, consider the path $P_n$ of order $n\geq 3$.  The reduced Wiener polynomial of $P_n$ is given by
\[
\widehat{W}(P_n;x)=\sum_{i=1}^{n-1} (n-i)x^{i-1}.
\]
Thus by the Enestr\"om-Kakeya Theorem, the nonzero Wiener roots of $P_n$ lie in the annulus $1<\frac{n-1}{n-2}\leq |z| \leq 2$.  This improves a result from~\cite{di}.

We now turn our attention to the collection of all Wiener roots of connected graphs of order $n$.  We find the exact maximum modulus among all roots in this collection, and we describe the unique graph with a Wiener root of this maximum modulus.

\begin{theorem}
The maximum modulus among Wiener roots of connected graphs of order $n \geq 2$ is $\binom{n}{2}-1$. Moreover, for $n \geq 3$, $K_n-e$ is the unique graph of order $n$ with a Wiener root of this modulus.
\end{theorem}

\begin{proof}
If $G$ is a complete graph, then $W(G;x) = nx$, whose only root is $0$.  So we may assume that $G$ is not complete, and hence has order $n\geq 3$ and diameter $D\geq 2$.  Let $z$ be any root of the reduced Wiener polynomial $\widehat{W}(G;x) =\sum_{i=1}^{D} d_ix^{i-1}$.  By the Enestr\"{o}m-Kakeya Theorem,
\[
|z| \le \max\left\{\tfrac{d_i}{d_{i+1}}: 1 \le i \le D-1\right\}.
\]
Since all the $d_{i}$ are at least $1$, $\sum_{i=1}^D d_i = \tbinom{n}{2}$, and $D \geq 2$, we observe that
\[
\max\left\{d_i: 1 \le i \le D\right\} \leq \tbinom{n}{2}-1 \ \ \ \mbox{ and } \ \ \ \min\left\{d_i : 1 \le i \le D\right\} \ge 1.
\]
Hence, by the Enestr\"{o}m-Kakeya Theorem, $|z| \le \binom{n}{2}-1$.  Finally, we note that
\[
W(K_{n}-e;x) = \left( \tbinom{n}{2}-1 \right) x + x^{2}
\]
has a root at $x =-\left(\binom{n}{2}-1\right)$.  Moreover, if $\max\{d_i/d_{i+1}\colon\ 1\leq i\leq D-1\} = \binom{n}{2}-1$, then $D = 2$, $d_{1} = \binom{n}{2}-1$ and $d_{2} = 1$, which implies that $G\cong K_{n}-e$.  We conclude that $K_n-e$ is the unique connected graph of order $n \geq 3$ with a Wiener root of modulus $\binom{n}{2}-1$.
\end{proof}

Thus, the growth of the maximum modulus among all Wiener roots of graphs of order $n$ is only $\Theta(n^2)$.  However, the problem is quite different for trees.  We show now that the growth rate of the maximum modulus among all Wiener roots of trees of order $n$ is $\Theta(n)$.  We begin with a lemma that will allow us to apply the Enestr\"{o}m-Kakeya Theorem directly.  Throughout, we use the fact that in any tree $T$, pairs at distance $k$ correspond exactly to paths of length $k$.

\begin{lemma}
\label{treelemmamodulus}
Let $T$ be a tree of order $n \ge 3$ with diameter $D$. Then
\[
\frac{d_k(T)}{d_{k+1}(T)} \le 2(n-D)
\]
for $1 \le k < D$.
\end{lemma}
\begin{proof}
We proceed by induction on $n$. If $n=3$, then $T\cong P_3,$ so $D=2,$ $d_1(T)=2,$ and $d_2(T)=1$. The result readily follows in this case.

Now assume, for some $n > 3$, that the result holds for all trees with  $n-1$ vertices. Let $T$ be a tree of order $n>3$ and with diameter $D$. Let $1 \le k < D$. A simple induction on $n$ shows that $d_2 \ge n-2$, so that
\[ \frac{d_1}{d_{2}} \leq \frac{n-1}{n-2} = 1 + \frac{1}{n-2} < 2 \leq 2(n-D),\]
so we can assume $k \ge 2$. There are two cases to consider.

\medbreak

\noindent
{\bf Case 1:} {\itshape There is a leaf $v$ on a longest path of $T$ such that $T-v$ has diameter $D$.}

\noindent
Thus, $D\leq n-2$.  Let $u$ be the neighbour of $v$ in $T$ and let $T' = T-v$. Let $d_{\ell}(T',u)$ denote the number of paths of length $\ell \ge 1$ in $T'$ that have $u$ as end-vertex. Observe that since $v$ was on a diametral path of $T$ (that is, a shortest path of length $\diam(G)$), there is some path of length $D-1$ extending from $u$ in $T'$, so $d_\ell(T',u)\geq 1$ for all $l\in\{1,2,\dots,D-1\}$.  Further, this implies that there are at most $n-1-(D-1)=n-D$ vertices at distance $\ell$ from $u$ in $T'$ for $\ell \in \{1,\ldots,D-1\}$.  All together, we have $1\leq d_\ell(T',u)\leq n-D$ for all $\ell\in\{1,2,\dots,D-1\}.$  Using this fact and the induction hypothesis we see that
\begin{align*}
d_k(T)&= d_k(T') + d_{k-1}(T',u)\\
& \le 2(n-1-D)d_{k+1}(T') + d_{k-1}(T',u)\\
&=2(n-1-D)\left[d_{k+1}(T)-d_k(T',u)\right]+d_{k-1}(T',u)\\
&=2(n-1-D)d_{k+1}(T) -2(n-1-D)d_k(T',u) + d_{k-1}(T',u)\\
&\le 2(n-1-D)d_{k+1}(T) -2(n-1-D) + n-D\\
&\le 2(n-1-D)d_{k+1}(T)-(n-D-2)\\
&\le 2(n-1-D)d_{k+1}(T)\\
& < 2(n-D)d_{k+1}(T).
\end{align*}
This holds for all $k$, $2 \le k < D$.

\medbreak

\noindent{\bf Case 2:} {\itshape If $v$ is any leaf vertex on a longest path of $T$, then $\diam(T-v) = D-1.$}

\noindent
It follows that there is a unique diametral path in $T$, say between leaves $u$ and $v$, and both $u$ and $v$ are adjacent to a vertex of degree $2$ in $T$.  Let $T'=T-v$, which has order $n-1$ and diameter $D-1$. If $2\le k \le D-2$, then the induction hypothesis applies to $T'$, and the argument is similar to that of Case 1:
\begin{align*}
d_k(T)&= d_k(T') + d_{k-1}(T',u)\\
& \le 2((n-1)-(D-1))d_{k+1}(T') + d_{k-1}(T',u)\\
&=2(n-D)\left[d_{k+1}(T)-d_k(T',u)\right]+d_{k-1}(T',u)\\
&=2(n-D)d_{k+1}(T) -2(n-D)d_k(T',u) + d_{k-1}(T',u)\\
&\le 2(n-D)d_{k+1}(T) -2(n-D) + n-D\\
&\le 2(n-D)d_{k+1}(T)-(n-D)\\
&< 2(n-D)d_{k+1}(T).
\end{align*}
Suppose now that $k=D-1$.  Since $T'$ has diameter $D-1$, the induction hypothesis does not apply as above.  However, we know that $d_D(T)=1,$ so it suffices to show that $d_{D-1}(T)\leq 2(n-D)$.

First, we show that every path of length $D-1$ in $T$ has an end in $\{u,v\}$.  Suppose otherwise that there is a pair of vertices $\{x,y\}$ at distance $D-1$ in $T$, with $\{x,y\}\cap\{u,v\}=\emptyset.$  Let $P:x = w_{1},w_{2},\ldots,w_{D}=y$ be the path between $x$ and $y$, let $P_{u}$ be the shortest path from $u$ to $P$, say meeting $P$ in vertex $w_{i}$ and similarly define $P_{v}$ as the shortest path from $v$ to $P$, meeting $P$ in vertex $w_{j}$; without loss, $i \leq j$ (see Figure~\ref{pathdiagram}). As the path between $u$ and $v$ is the unique diametral path, it must be the case that the path $P_{u}$ is at most the length of the path between $x$ and $w_{i}$ (otherwise, the path between $u$ and $y$ would be a second diametral path). Likewise, the path $P_{v}$ is at most the length between $v$ and $w_{j}$. However, then the path between $u$ and $v$ would have length at most that of $P$, which is $D-1$, a contradiction.

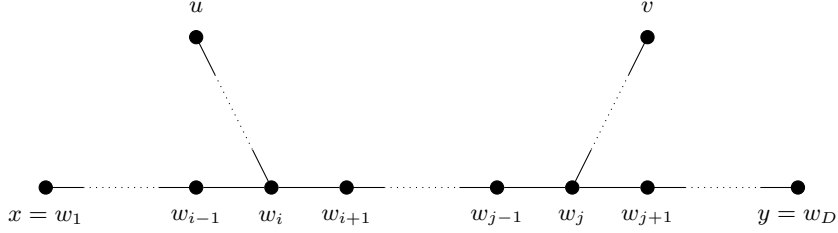
\begin{figure}[ht]
\centering
\begin{tikzpicture}
\vertex (1) at (0,0) {};
\vertex (2) at (2,0) {};
\vertex (3) at (3,0) {};
\vertex (4) at (4,0) {};
\vertex (5) at (6,0) {};
\vertex (6) at (7,0) {};
\vertex (7) at (8,0) {};
\vertex (8) at (10,0) {};
\vertex (9) at (2,2) {};
\vertex (10) at (8,2) {};
\path
(2) edge (3)
(3) edge (4)
(5) edge (6)
(6) edge (7);
\draw (0,0) -- (0.5,0);
\draw[dotted] (0.5,0) -- (1.5,0);
\draw (1.5,0) -- (2,0);
\draw (4,0) -- (4.5,0);
\draw[dotted] (4.5,0) -- (5.5,0);
\draw (5.5,0) -- (6,0);
\draw (8,0) -- (8.5,0);
\draw[dotted] (8.5,0) -- (9.5,0);
\draw (9.5,0) -- (10,0);
\draw (2,2) -- (2.25,1.5);
\draw[dotted] (2.25,1.5) -- (2.75,0.5);
\draw (2.75,0.5) -- (3,0);
\draw (8,2) -- (7.75,1.5);
\draw[dotted] (7.75,1.5) -- (7.25,0.5);
\draw (7.25,0.5) -- (7,0);
\draw (0,-0.4) node {\footnotesize $x=w_1$};
\draw (2,-0.4) node {\footnotesize $w_{i-1}$};
\draw (3,-0.4) node {\footnotesize $w_{i}$};
\draw (4,-0.4) node {\footnotesize $w_{i+1}$};
\draw (6,-0.4) node {\footnotesize $w_{j-1}$};
\draw (7,-0.4) node {\footnotesize $w_{j}$};
\draw (8,-0.4) node {\footnotesize $w_{j+1}$};
\draw (10,-0.4) node {\footnotesize $y=w_D$};
\draw (2,2.4) node {\footnotesize $u$};
\draw (8,2.4) node {\footnotesize $v$};
\end{tikzpicture}
\caption{The shortest $x$--$y$ path meets the shortest $u$--$v$ path.}
\label{pathdiagram}
\end{figure}

Thus every pair of vertices at distance $D-1$ in $T$ contains either $u$ or $v$.  But since there is a path of length $D$ between $u$ and $v$, and exactly one vertex on this path is at distance $D-1$ from $u$, there are at most $n-D$ vertices at distance $D-1$ from $u$.  Similarly, there are at most $n-D$ vertices at distance $D-1$ from $v$.  Therefore, $d_{D-1}(T)\leq 2(n-D),$ as desired.
\end{proof}

Note that Lemma~\ref{treelemmamodulus} implies immediately that $\tfrac{d_k(T)}{d_{k+1}(T)}\leq 2(n-2)$ for any tree $T$ of order $n\geq 3$ and any $1\leq k<D(T)$, since $D(T)\geq 2$ for any tree of order at least $3$.  We can prove a stronger result for trees of order $n\geq 5$. We use the following notation. For each $2\leq i\leq \lfloor\tfrac{n}{2}\rfloor$, let $D_{i,n-i}$ denote the double star of order $n$ with $i-1$ leaves attached to one internal vertex and $n-i-1$ leaves attached to the other (see Figure~\ref{doublestar}). Note that all double stars of order $n$ have this form for some $2\leq i\leq \lfloor\tfrac{n}{2}\rfloor$.

\begin{figure}[ht]
\centering
\begin{tikzpicture}
\vertex (1) at (1,0) {};
\vertex (2) at (2,0) {};
\vertex (3) at (0,0.5) {};
\vertex (4) at (0,0.2) {};
\vertex (5) at (0,-0.5) {};
\vertex (6) at (3,0.7) {};
\vertex (7) at (3,0.4) {};
\vertex (8) at (3,-0.7) {};
\path
(1) edge (2)
(1) edge (3)
(1) edge (4)
(1) edge (5)
(2) edge (6)
(2) edge (7)
(2) edge (8);
\draw[dotted] (0,0) -- (0,-0.3);
\draw[dotted] (3,0.1) -- (3,-0.4);
\draw [decorate,decoration={brace,amplitude=4pt,mirror},xshift=-4pt,yshift=0pt]
(0,0.65) -- (0,-0.65) node [left,black,midway,xshift=-3pt]
{\footnotesize $i-1$};
\draw [decorate,decoration={brace,amplitude=4pt},xshift=4pt,yshift=0pt]
(3,0.85) -- (3,-0.85) node [right,black,midway,xshift=3pt]
{\footnotesize $n-i-1$};
\end{tikzpicture}
\caption{The double star $D_{i,n-i}$.}
\label{doublestar}
\end{figure}
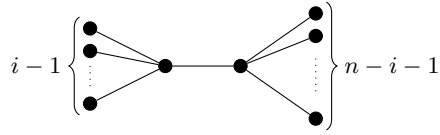

\begin{corollary}\label{UpperBound}
Let $T$ be a tree of order $n\geq 5$.  Then
\[
\tfrac{d_k}{d_{k+1}}\leq 2(n-4)
\]
for $1\leq k<D,$ where $D$ is the diameter of $T$.
\end{corollary}

\begin{proof}
By Lemma~\ref{treelemmamodulus},
\[
\tfrac{d_k}{d_{k+1}}\leq 2(n-D)
\]
for all $1\leq k<D$, so if $D\geq 4$ then we are done.  Since $n\geq 5,$ we only need to consider the cases $D=2$ and $D=3.$

\medbreak

\noindent{\bf Case 1: }{\itshape $D=2$.}

\noindent
 Then $T\cong K_{1,n-1},$ and
\[
W(T;x)=(n-1)x+\binom{n-1}{2}x^2.
\]
Thus $\tfrac{d_1}{d_2}=\tfrac{2}{n-2}<2(n-4)$ for all $n\geq 5.$

\medbreak

\noindent{\bf Case 2: }{\itshape $D=3$.}

\noindent Then $T=D_{i, n-i}$ for some $2 \le i \le \lfloor \frac{n}{2} \rfloor$. 
By straightforward counting, we find
\[
W(D_{i,n-i};x)=(n-1)x+\left[\binom{i}{2}+\binom{n-i}{2}\right]x^2+(i-1)(n-i-1)x^3.
\]
Thus
\[
\frac{d_1}{d_2}=\frac{n-1}{\binom{i}{2}+\binom{n-i}{2}}<\frac{n-1}{2}\leq 2(n-4) \ \ \ \ \ \mbox{ for } n\geq 5,
\]
and
\[
\frac{d_2}{d_3}=\frac{\binom{i}{2}+\binom{n-i}{2}}{(i-1)(n-i-1)}\leq \frac{2\binom{n-i}{2}}{(i-1)(n-i-1)}=\frac{n-i}{i-1}\leq n-2\leq 2(n-4) \ \ \mbox{ for } n\geq 6.
\]
Finally, when $n=5$, we necessarily have $i=2$, and we compute directly that $\tfrac{d_2}{d_3}=2,$ which satisfies the desired bound.
\end{proof}

Now for any $n\geq 5$, the upper bound of Corollary~\ref{UpperBound} is also an upper bound on the modulus of any Wiener root of a tree of order $n$ by the Enestr\"om-Kakeya Theorem.

\begin{corollary}\label{TreeRootBound}
Let $T$ be a tree of order $n\geq 5$.  If $z$ is a Wiener root of $T$ then $|z|\leq 2(n-4)$.  \hfill \qed
\end{corollary}

While we have observed that the bound of Corollary~\ref{TreeRootBound} is not tight for small values of $n$, we demonstrate in the next theorem that there is a tree $T_n$ (see Figure~\ref{treeexamplemaxmod}) of order $n$ that has a Wiener root whose modulus approaches $\left(1+\tfrac{1}{\sqrt{2}}\right)n$ asymptotically.  This means that the growth of the maximum modulus, among all Wiener roots of trees of order $n$, is  $\Theta(n)$.  We have verified that $T_n$ is the unique tree with Wiener root of maximum modulus among all trees of order $n$ for $5\leq n\leq 17,$ and suspect that this is true for all $n\geq 5.$

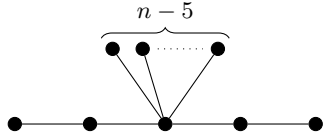
\begin{figure}[ht]
\centering
\begin{tikzpicture}
\vertex (1) at (0,0) {};
\vertex (2) at (1,0) {};
\vertex (3) at (2,0) {};
\vertex (4) at (3,0) {};
\vertex (5) at (4,0) {};
\vertex (6) at (1.3,1) {};
\vertex (7) at (1.7,1) {};
\vertex (8) at (2.7,1) {};
\path
(1) edge (2)
(2) edge (3)
(3) edge (4)
(4) edge (5)
(3) edge (6)
(3) edge (7)
(3) edge (8);
\draw[dotted] (1.9,1) -- (2.5,1);
\draw [decorate,decoration={brace,amplitude=4pt},yshift=4pt]
(1.15,1) -- (2.85,1) node [above,black,midway,yshift=3pt]
{\footnotesize $n-5$};
\end{tikzpicture}
\caption{The tree $T_{n}$.}
\label{treeexamplemaxmod}
\end{figure}

\begin{theorem}
For $n \geq 6$, the tree $T_n$ has a real Wiener root in the interval
\[
\left(-\left( 1 + \tfrac{1}{\sqrt{2}}\right)n+7, -\left( 1 + \tfrac{1}{\sqrt{2}}\right)n + 8\right).
\]
\end{theorem}
\begin{proof}

Observe that
\begin{align*}
\widehat{W}(T_{n};x) & = n-1 + \left( \tbinom{n-3}{2} +2 \right) x + 2(n-4) x^{2} + x^{3}
\end{align*}
In particular,
\[
\widehat{W}\left(T_{n};-\left( 1 + \tfrac{1}{\sqrt{2}}\right)n + 7\right) = \left( \frac{3}{4} \sqrt{2}-\frac{3}{2}\right)n^2+\left(-\frac{43}{2} \sqrt{2}+\frac{63}{2} \right)n+6,
\]
which is negative for $n \geq 6$, and
\[
\widehat{W}\left(T_{n};-\left( 1 + \tfrac{1}{\sqrt{2}}\right)n + 8\right)  = \left( \frac{7}{4}\sqrt{2}-\frac{1}{2}\right)n^2+(-36\sqrt{2}+29)n+63,
\]
which is positive for all $n$. By the Intermediate Value Theorem, $\widehat{W}(T_{n};x)$, and hence $W(T_n;x)$, has a root in the interval $\left(-\left( 1 + \tfrac{1}{\sqrt{2}}\right)n+7, -\left( 1 + \tfrac{1}{\sqrt{2}}\right)n + 8\right)$ for all $n \geq 6$.
\end{proof}

Having obtained an upper bound on the modulus of any Wiener root of a graph (or tree) of order $n$, we now turn to a lower bound.  Of course, the Wiener root of smallest modulus among all connected graphs of order $n$ is $0$, as $0$ is a root of every Wiener polynomial, but what about the {\em nonzero} root of smallest modulus? Equivalently, what is the smallest modulus of a root of the reduced Wiener polynomial of a connected graph of order $n$?  Below we prove a lower bound on the ratio of successive coefficients of the Wiener polynomial which allows us to apply the Enestr\"om-Kakeya Theorem once more.

\begin{lemma}\label{RatioLowerBound}
Let $G$ be a connected graph of order $n\geq 3$ and diameter $D\geq 2$.  Then
\[
\frac{d_k}{d_{k+1}} \geq \frac{2}{n-k-1}
\]
for $1\leq k<D.$
\end{lemma}

\begin{proof}
Let $k\in\{1,\dots,D-1\}.$
We count ordered pairs $(w,\{u,v\})$, where $d(u,v)=k$ and either $w$ is joined to $u$ and at distance $k+1$ from $v$, or $w$ is joined to $v$ and at distance $k+1$ from $u$; let ${\mathcal S}_{k}$ be the set of all such ordered pairs. Clearly, $|S_k|$ is at most $(n-(k+1))d_{k} = (n-k-1)d_{k}$, as for any pair of vertices $\{u,v\}$ at distance $k$, the $k+1$ vertices on any shortest path between $u$ and $v$ cannot serve as the point $w$.

On the other hand, for any pair of vertices $x$ and $y$ at distance $k+1$, fix a path $P$ of this length between them, with vertex $x_1$ adjacent to $x$ on $P$, and $y_1$ adjacent to $y$ on $P$ (it may be the case that $x_1 = y_1$, if $k=1$). Then we can form two distinct ordered pairs, $(x,\{x_1,y\})$ and  $(y,\{x,y_1\})$, both of which belong to ${\mathcal S}_{k}$.  Working over all pairs of vertices at distance $k+1$, we generate $2d_{k+1}$ pairs through this process, and we claim that they are all distinct.  Note first of all that the first coordinate in any ordered pair generated in this manner is adjacent to precisely one element of the second coordinate.  Suppose that two pairs $\{x,y\}$ and $\{x',y'\}$ at distance $k+1$ yield the same member of ${\mathcal S}_k$, say $(x,\{x_1,y\})$.  Then either $x'=x$ or $y' = x$.  Without loss of generality, assume that $x'=x$.  Since $x$ is adjacent to $x_1$ and $d(x',y')=k+1>1$, it follows that $y'\neq x_1,$ and thus $y'=y$. We conclude that
\[
(n-k-1)d_{k} \geq |{\mathcal S}_{k}| \geq 2d_{k+1}. \qedhere
\]
\end{proof}

\begin{corollary}
Among all connected graphs of order $n \geq 3$, the graph $K_{1,n-1}$ has nonzero Wiener root of mimimum modulus $\tfrac{2}{n-2}$. Moreover, $K_{1,n-1}$ is the unique graph of order $n$ with a Wiener root of this modulus for all $n\geq 3$.
\end{corollary}

\begin{proof}
Since $W(K_{1,n-1};x)=(n-1)x+\binom{n-1}{2}x^2,$ the star $K_{1,n-1}$ has a single nonzero Wiener root $-\tfrac{2}{n-2}.$

Now let $G$ be a connected graph of order $n$ and size $m$.  Let
\[
r=\min\left\{\tfrac{d_k}{d_{k+1}}\colon\ 1\leq k<D\right\}.
\]
From Lemma~\ref{RatioLowerBound}, $\tfrac{d_k}{d_{k+1}} \geq \tfrac{2}{n-k-1}\geq \tfrac{2}{n-2}$ for all $1\leq k<D.$  Thus, by the Enestr\"om-Kakeya Theorem, if $z$ is a nonzero Wiener root of $G$, then
\[
|z|\geq r \geq \tfrac{2}{n-2}.
\]
To see that the star is the unique connected graph of order $n$ with a Wiener root of modulus $\tfrac{2}{n-2},$ it suffices to show that if $G\not\cong K_{1,n-1},$ then $r>\tfrac{2}{n-2}$.  By Lemma \ref{RatioLowerBound},
\[
\tfrac{d_k}{d_{k+1}}\geq \tfrac{2}{n-k-1}\geq \tfrac{2}{n-3} \ \ \ \ \mbox{ for all } 2\leq k<D.
\]
It remains only to show that $\tfrac{d_1}{d_2}>\tfrac{2}{n-2}$ if $G\not\cong K_{1,n-1}$.  If $G$ is not a tree, then $d_1\geq n$ and $d_2\leq \binom{n}{2}-n=\tfrac{n(n-3)}{2}$, so
\[
\tfrac{d_1}{d_2}\geq \frac{n}{\tfrac{n(n-3)}{2}}=\tfrac{2}{n-3}>\tfrac{2}{n-2}.
\]
If $G$ is a tree but $G\not\cong K_{1,n-1},$ then $G$ has diameter at least $3$.  So $d_1=n-1$ and $d_2<\binom{n}{2}-(n-1)=\binom{n-1}{2}.$  Hence,
\[
\tfrac{d_1}{d_2}>\frac{n-1}{\tfrac{(n-1)(n-2)}{2}}=\tfrac{2}{n-2}.
\]
This completes the proof.
\end{proof}

\section{Density of Real Wiener Roots}\label{DensitySection}

From the work of Jackson~\cite{Jackson} and Thomassen~\cite{Thom}, it is known that the closure of the collection of real chromatic roots of all graphs is $\{0,1\} \cup [32/27, \infty)$. In \cite{BHN} it is shown that the closure of the real roots of independence polynomials of graphs is $(-\infty, 0]$, and in~\cite{BC} it is proven that the closure of the collection of all real roots of all-terminal reliability polynomials is $\{0\} \cup [1,2]$.  Since the coefficients of the Wiener polynomial of any connected graph are all positive, any real Wiener root lies in $(-\infty, 0]$.  The next result addresses the density of the real Wiener roots in this interval.

\begin{theorem}\label{Dense}
The collection of all real Wiener roots is dense in $(-\infty, 0]$, even when restricted to graphs of diameter $2$.
\end{theorem}
\begin{proof}
Let $m,n$ be positive integers where $n\geq 3$ and $n-1 \le m < \binom{n}{2}$. Then there is a graph of order $n$, size $m$ and of diameter $2$. Such a graph can be obtained by starting with the star $K_{1,n-1}$ and adding edges between any $m-n+1$ non-adjacent pairs of vertices. The Wiener polynomial for any graph $H$ constructed in this way is given by
\[
W(H;x) =mx +\left(\tbinom{n}{2}-m\right) x^2.
\]
This Wiener polynomial has root $\frac{-m}{\binom{n}{2} -m}$.  We now show that the collection of all such roots is dense in $(-\infty, 0].$

Let $a,b \in {\mathbb{N}}$. We construct a graph of diameter 2 with Wiener root $-\frac{a}{b}$.  Since the negative rationals are dense in $(-\infty,0],$ the result follows immediately.  Let $n=2(a+b)$ and $m =a(2(a+b) -1)$.  Note that $n\geq 3$ and $n-1\leq m<\binom{n}{2}$.  By the above observation, there is a graph of diameter $2$ on $n$ vertices and $m$ edges, which has Wiener root
\[
\frac{-m}{\binom{n}{2} -m}=\frac{a(2(a+b)-1)}{(a+b)(2(a+b)-1)-a(2(a+b)-1)}=-\frac{a}{b}.\qedhere
\]
\end{proof}

\bigbreak

Next we consider a similar problem for trees: is the collection of real Wiener roots of trees dense in $(-\infty,0]$?  While this problem appears more difficult, we prove that the closure of the collection of real Wiener roots of trees contains $(-\infty,-1]$.  We will use the family of double stars in our proof, so we first prove that the Wiener roots of all double stars of sufficiently large order are all real.  Recall that for $n\geq 4$ and $2\leq k\leq \lfloor n/2\rfloor$, $D_{k,n-k}$ denotes the double star of order $n$ with two internal vertices adjacent to $k-1$ and $n-k-1$ leaves, respectively (see Figure~\ref{doublestar}).

\begin{lemma}\label{DoubleStarReal}
Let $n\geq 15$ and $2\leq k\leq \lfloor n/2\rfloor$.  Then the Wiener roots of the double star $D_{k,n-k}$ are all real.
\end{lemma}

\begin{proof}
The Wiener polynomial of $D_{k,n-k}$ is given by
\[
W(D_{k,n-k};x)=(n-1)x + \left(\tbinom{k}{2}+ \tbinom{n-k}{2}\right)x^2 + (k-1)(n-k-1) x^3.
\]
The discriminant of the reduced Wiener polynomial $\widehat{W}(D_{k,n-k};x)$ is given by 
\[\left(\tbinom{k}{2}+ \tbinom{n-k}{2}\right)^2-4(n-1)(k-1)(n-k-1).\] By applying elementary Calculus to each of the two terms in this expression we see that
\begin{align*}
&\left(\tbinom{k}{2}+ \tbinom{n-k}{2}\right)^2-4(n-1)(k-1)(n-k-1)  \\
&\geq  \left( \left(\tfrac{n}{2} \right)\left(\tfrac{n}{2} -1 \right) \right)^2 - 4(n-1)
 \left(\tfrac{n}{2}-1\right)\left(n-\tfrac{n}{2}-1\right)\\
& =  \tfrac{1}{16}(n^2-16n+16)(n-2)^2\\
& >  0.
\end{align*}
Thus we conclude that the Wiener roots of any double star of order $n\geq 15$ are all real.
\end{proof}

Using a computer algebra system it was verified that there are double stars of each order $4\leq n\leq 14$ with nonreal Wiener roots; so we cannot avoid the assumption that $n\geq 15$ in the previous lemma.  We are now ready to prove our result concerning the closure of the collection of real Wiener roots of trees.

\begin{theorem}\label{TreeDense}
The closure of the collection of real Wiener roots of trees contains the interval $(-\infty, -1]$.
\end{theorem}
\begin{proof}
Let $r =\frac{a}{b}$ where $a, b \in \mathbb{N}$.  Let $n = (2a+b)\ell$ for some $\ell \in \mathbb{N}$ with $\ell\geq 5$ (so that $n\geq 15$), and define $k=\frac{n}{2r+1}.$  Then
\[
k=\frac{n}{2r+1}=\frac{(2a+b)\ell}{2\tfrac{a}{b}+1}=\ell b.
\]
So $k$ is an integer and certainly $2\leq k\leq n-2$.  The Wiener polynomial of the double star $D_{k,n-k}$ is given by
\[
W(D_{k,n-k};x)=(n-1)x + \left(\tbinom{k}{2}+ \tbinom{n-k}{2}\right)x^2 + (k-1)(n-k-1) x^3.
\]
From Lemma~\ref{DoubleStarReal}, all of the roots of $D_{k,n-k}$ are real. Substituting $k=\frac{n}{2r+1}$ into this polynomial, and solving for the nonzero roots using the quadratic formula, and then taking the limit of the left-most root as $\ell \rightarrow \infty$, we see that this root approaches $-r-\frac{1}{4r}$.

Therefore, the set of limit points of real Wiener roots of double stars contains the set $\left\{-r-\tfrac{1}{4r}|r \in \mathbb{Q}^+\right\}$.  Using elementary calculus it can be shown that the real-valued function $f(x) = -x -\frac{1}{4x}$, for $x > 0$,  attains its maximum value of $-1$ at $r = \frac{1}{2}$.  We see immediately that there is a double star with Wiener root as close as we wish to $-1$.  Moreover, $\lim_{x \rightarrow \infty}f(x) = -\infty$.  Thus for every $y \in (-\infty, -1)$  there is  an $x \in (1/2, \infty)$ such that $f(x)=y$. Since $f$ is continuous on $(1/2, \infty)$ there is for every $\epsilon > 0$ a $\delta >0$ such that if $\overline{x} \in (x- \delta, x+\delta)$, then $|y-f(\overline{x})| < \epsilon$. Since the rational numbers are dense in the real numbers, there is a rational $r \in (x-\delta, x+\delta)$.  Thus $|y-f(r)| < \epsilon$.  Hence there is a double star with a root as close as we wish to a given real number in $(-\infty, -1]$.
\end{proof}

\section{Imaginary and Real Parts of Wiener Roots}\label{RealAndImagSection}

We now consider the collection of all \emph{complex} Wiener roots of graphs, and focus on the imaginary and real parts. We have already seen that there are real Wiener roots of arbitrarily large modulus, but we have only seen examples of arbitrarily large negative real  Wiener roots.  In this section we show that there are Wiener roots with arbitrarily large positive real part and that there are Wiener roots with arbitrarily large imaginary part (even if we restrict to trees).
The family of ``brooms'' is useful for proving these results. For integers $n > k \ge 3$ the {\em broom} $B_{k,n-k}$ of order $n$ having a {\em handle} of order $k$ is obtained from a path of order $k$ by joining $n-k$ leaves to exactly one of its end-vertices.

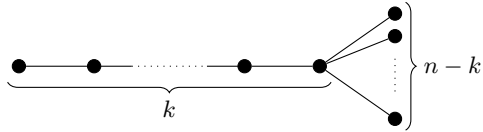
\begin{figure}[ht]
\centering
\begin{tikzpicture}
\vertex (1) at (0,0) {};
\vertex (2) at (1,0) {};
\vertex (3) at (3,0) {};
\vertex (4) at (4,0) {};
\vertex (5) at (5,0.7) {};
\vertex (6) at (5,0.4) {};
\vertex (7) at (5,-0.7) {};
\path
(1) edge (2)
(3) edge (4)
(4) edge (5)
(4) edge (6)
(4) edge (7);
\draw (1,0) -- (1.5,0);
\draw[dotted] (1.5,0) -- (2.5,0);
\draw (2.5,0) -- (3,0);
\draw[dotted] (5,0.1) -- (5,-0.4);
\draw [decorate,decoration={brace,amplitude=4pt},xshift=4pt,yshift=0pt]
(5,0.85) -- (5,-0.85) node [right,black,midway,xshift=3pt]
{\footnotesize $n-k$};
\draw [decorate,decoration={brace,amplitude=4pt,mirror},xshift=0pt,yshift=-6pt]
(-0.15,0) -- (4.15,0) node [below,black,midway,yshift=-3pt]
{\footnotesize $k$};
\end{tikzpicture}
\caption{The broom $B_{k,n-k}$.}
\label{broom}
\end{figure}

\begin{proposition}\label{LargeImaginaryPart}
There are connected graphs (even trees) with Wiener roots having arbitrarily large imaginary part.
\end{proposition}

\begin{proof}
For $n\geq 4$, let $G_n$ denote the graph obtained from $K_{n-1}-e$ by attaching a leaf to one of the vertices of degree $n-3$.  Then
\[
W(G_n;x)=\binom{n-1}{2}x+(n-2)x^2+x^3,
\]
which has nonzero roots
\[
-\tfrac{n}{2}+1\pm\tfrac{1}{2}\sqrt{2n-n^2},
\]
whose imaginary parts tend to $\pm\tfrac{n}{2}$ asymptotically.

It remains to show that there are trees with Wiener roots having arbitrarily large imaginary part.  We consider a collection of brooms to achieve this result.  Using straightforward counting techniques, we find
\[
W(B_{4,n-4};x)=(n-1)x+\left(\tbinom{n-3}{2}+2\right)x^2+(n-3)x^3+(n-4)x^4.
\]
Using a computer algebra system, we obtain explicit expressions for the roots of the reduced Wiener polynomial $\widehat{W}(B_{4,n-4};x)$ in terms of $n$.  These include a pair of nonreal roots $a_n\pm b_ni$ for which
\[
\lim_{n\rightarrow \infty}\tfrac{b_n}{\sqrt{n}}=2^{-1/2}.\qedhere
\]
\end{proof}

It follows that the largest imaginary part of a Wiener root of a connected graph of order $n$ is $\Omega(n),$ while the largest imaginary part of a Wiener root of a tree of order $n$ is only known to be $\Omega(\sqrt{n})$.

We now consider the positive real part of Wiener roots.  We first summarize what we know for connected graphs and trees of small order; this information has been verified using a computer algebra system.  For $n\leq 5$, we found that there are no graphs of order $n$ with Wiener roots having positive real part.  For $6\leq n\leq 9,$ the connected graph with Wiener root of largest positive real part is the path $P_n,$ and in fact, for $6\leq n\leq 15$, the tree with Wiener root of largest positive real part is $P_n$.  We have already mentioned that the Wiener roots of $P_n$ all have modulus at most $2$ (and hence real part at most $2$).  However, we found that the trees of order $n$ with the Wiener root of largest positive real part are not paths for $n=16$ and $n=17$ (see Figure~\ref{Tree16} and Figure~\ref{Tree17}, respectively).  We show below that the largest positive real part of a Wiener root of a connected graph (or tree) of order $n$ has growth rate $\Omega(\sqrt[3]{n}).$  As for Proposition~\ref{LargeImaginaryPart}, a particular family of brooms is used.

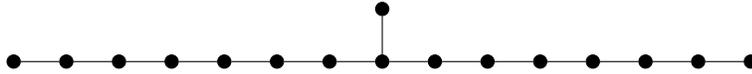
\begin{figure}[htb]
\centering
\begin{tikzpicture}[scale=0.7]
\vertex (1) at (0,0) {};
\vertex (2) at (1,0) {};
\vertex (3) at (2,0) {};
\vertex (4) at (3,0) {};
\vertex (5) at (4,0) {};
\vertex (6) at (5,0) {};
\vertex (7) at (6,0) {};
\vertex (8) at (7,0) {};
\vertex (9) at (8,0) {};
\vertex (10) at (9,0) {};
\vertex (11) at (10,0) {};
\vertex (12) at (11,0) {};
\vertex (13) at (12,0) {};
\vertex (14) at (13,0) {};
\vertex (15) at (14,0) {};
\vertex (16) at (7,1) {};
\path
(1) edge (2)
(2) edge (3)
(3) edge (4)
(4) edge (5)
(5) edge (6)
(6) edge (7)
(7) edge (8)
(8) edge (9)
(9) edge (10)
(10) edge (11)
(11) edge (12)
(12) edge (13)
(13) edge (14)
(14) edge (15)
(16) edge (8);
\end{tikzpicture}
\caption{The tree of order $16$ with a Wiener root of largest positive real part.}
\label{Tree16}
\end{figure}

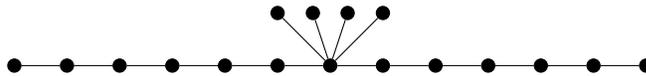
\begin{figure}[htb]
\centering
\begin{tikzpicture}[scale=0.7]
\vertex (1) at (0,0) {};
\vertex (2) at (1,0) {};
\vertex (3) at (2,0) {};
\vertex (4) at (3,0) {};
\vertex (5) at (4,0) {};
\vertex (6) at (5,0) {};
\vertex (7) at (6,0) {};
\vertex (8) at (7,0) {};
\vertex (9) at (8,0) {};
\vertex (10) at (9,0) {};
\vertex (11) at (10,0) {};
\vertex (12) at (11,0) {};
\vertex (13) at (12,0) {};
\vertex (14) at (5,1) {};
\vertex (15) at (5.67,1) {};
\vertex (16) at (6.33,1) {};
\vertex (17) at (7,1) {};
\path
(1) edge (2)
(2) edge (3)
(3) edge (4)
(4) edge (5)
(5) edge (6)
(6) edge (7)
(7) edge (8)
(8) edge (9)
(9) edge (10)
(10) edge (11)
(11) edge (12)
(12) edge (13)
(14) edge (7)
(15) edge (7)
(16) edge (7)
(17) edge (7);
\end{tikzpicture}
\caption{The tree of order $17$ with a Wiener root of largest positive real part.}
\label{Tree17}
\end{figure}

\begin{proposition}\label{LargeReal}
There are connected graphs (even trees) with Wiener roots having arbitrarily large positive real part.
\end{proposition}

\begin{proof}
Using straightforward counting techniques, we find
\[
W(B_{5,n-5};x)=(n-1)x+\left(\tbinom{n-4}{2}+3\right)x^2+(n-3)x^3+(n-4)x^4+(n-5)x^5.
\]
Using a computer algebra system, we obtain an explicit expression for the roots of the reduced Wiener polynomial $\widehat{W}(B_{5,n-5};x)$ in terms of $n$.  These include a pair of nonreal roots $a_n\pm b_ni$ for which
\[
\lim_{n\rightarrow \infty}\tfrac{a_n}{\sqrt[3]{n}}=2^{-4/3}. \qedhere
\]
\end{proof}

Taken together, Theorem~\ref{TreeDense}, Proposition~\ref{LargeImaginaryPart}, and Proposition~\ref{LargeReal} imply that the collection of Wiener roots of all connected graphs (or trees) is not contained in any \emph{half-plane}; a region in the complex plane consisting of all points on one side of a line (and no points on the other side).  This is at least weak evidence that the collection of all complex Wiener roots may be dense in the entire complex plane, even if we only consider trees.

\begin{theorem}
The collection of complex Wiener roots of all graphs (or trees) is not contained in any half-plane.
\end{theorem}

\begin{proof}
Suppose otherwise that the Wiener roots of all trees are contained in some half-plane $H$.  If $H$ is described by a horizontal line, this contradicts Proposition~\ref{LargeImaginaryPart}, since complex roots of polynomials with real coefficients come in conjugate pairs.  Otherwise, $H$ is described by a line that intersects the real axis.  We have two cases.

\medbreak

\noindent
{\bf Case 1: }{\itshape Some interval $(-\infty,a)$ of the real line is not contained in $H$.}

\noindent
This contradicts Theorem~\ref{TreeDense}.

\medbreak

\noindent
{\bf Case 2: }{\itshape Some interval $(b,\infty)$ of the real line is not contained in $H$.}

\noindent
Then at least one of the quarter planes $\{z\in \mathbb{C}\colon\ \Re(z)>b \mbox{ and } \Im(z)\geq 0\}$ or $\{z\in\mathbb{C}\colon\ \Re(z)>b \mbox{ and } \Im(z)\leq 0\}$ is not contained in the half-plane.  Since complex roots of polynomials with real coefficients come in conjugate pairs, this contradicts Proposition~\ref{LargeReal}.
\end{proof}

\medbreak

We close this section with some interesting observations.  Firstly, we note that there are graphs with purely imaginary Wiener roots.  Figure~\ref{smallpurelyimaginary} contains the graph of smallest order that has a purely imaginary Wiener root, and Figure~\ref{treewieneri} shows the tree of smallest order with $i$ as a Wiener root.  This observation is interesting in light of the fact that although chromatic roots are known to be dense in the complex plane, it is widely suspected that there are no purely imaginary chromatic roots~\cite{Bohn}.  Moreover, it is not known whether there are independence or reliability polynomials with purely imaginary roots.

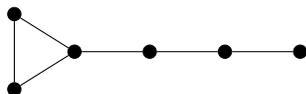
\begin{figure}[htb]
\centering
\begin{tikzpicture}
\vertex (1) at (0,0) {};
\vertex (2) at (1,0) {};
\vertex (3) at (2,0) {};
\vertex (4) at (3,0) {};
\vertex (5) at (-0.8,0.5) {};
\vertex (6) at (-0.8,-0.5) {};
\path
(1) edge (2)
(2) edge (3)
(3) edge (4)
(1) edge (6)
(1) edge (5)
(5) edge (6);
\end{tikzpicture}
\caption{Smallest graph with a purely imaginary Wiener root.}
\label{smallpurelyimaginary}
\end{figure}

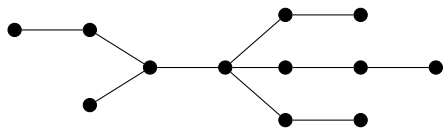
\begin{figure}[htb]
\centering
\begin{tikzpicture}
\vertex (1) at (0,0) {};
\vertex (2) at (1,0) {};
\vertex (3) at (1.8,0.7) {};
\vertex (4) at (1.8,0) {};
\vertex (5) at (-0.8,0.5) {};
\vertex (6) at (-0.8,-0.5) {};
\vertex (7) at (1.8,-0.7) {};
\vertex (8) at (-1.8,0.5) {};
\vertex (9) at (2.8,0.7) {};
\vertex (10) at (2.8,0) {};
\vertex (11) at (3.8,0) {};
\vertex (12) at (2.8,-0.7) {};
\path
(1) edge (2)
(2) edge (3)
(2) edge (4)
(1) edge (6)
(1) edge (5)
(2) edge (7)
(5) edge (8)
(3) edge (9)
(4) edge (10)
(10) edge (11)
(7) edge (12);
\end{tikzpicture}
\caption{Smallest tree with $i$ as a Wiener root.}
\label{treewieneri}
\end{figure}

Finally, we discuss graphs whose Wiener roots are all real.  The property of having all real roots is of interest combinatorially, since if the roots of a (real) polynomial with positive coefficients are all real, then the sequence of coefficients is log-concave, and hence unimodal -- see \cite{Stanley}, for example.  This condition has facilitated several proofs that certain graph theoretical sequences are log-concave.  For example, it was used to prove that the sequence of coefficients of the independence polynomial (the number of independent sets of each cardinality) of any clawfree graph is log-concave \cite{Chudnovsky}.

The Wiener roots of any graph of diameter $2$ are all real (as the corresponding reduced Wiener polynomial is linear).  Since it is well-known that almost every random graph $G \in {\mathcal G}(n,p)$ for fixed $p \in (0,1)$ has diameter $2$, we conclude that the Wiener roots of almost all graphs are real.  On the other hand, we have seen that there are infinitely many graphs with nonreal Wiener roots.  In contrast, for other graph polynomials, either almost all graphs have a nonreal root (as is the case for the chromatic~\cite{BW} and independence polynomial~\cite{BN2}) or the roots are always real (as is the case for the matching polynomial~\cite{HL}).

\section{Open Problems}

We conclude with a number of open problems. First, while we have maximized the modulus of Wiener roots of connected graphs of order $n$, and found the extremal graphs, our solution for trees is only asymptotic.

\begin{problem}
What is the tree of order $n$ with the Wiener root of largest modulus?
\end{problem}

\noindent Calculations suggest that the extremal graph is the tree $T_{n}$ shown in Figure~\ref{treeexamplemaxmod}.

\medbreak

While we have determined the closure of the real Wiener roots of graphs in general, we have not completed the discussion for trees. We propose the following.

\begin{problem}
Is the closure of real Wiener roots of trees the entire interval $(-\infty,0]$?
\end{problem}

\noindent
While we have found many Wiener roots in the interval $(-1,0)$, we have been unable to prove that the closure contains this interval.

On a related note, the closure of the collection of all complex Wiener roots is intriguing.  We have observed that the Wiener roots of almost all graphs are real, however, we also demonstrated that the collection of complex Wiener roots is not contained in any half-plane, even if we restrict ourselves to trees.

\begin{problem}
Is the closure of the collection of Wiener roots of connected graphs (or trees) the entire complex plane?
\end{problem}

\noindent
At this time, we do not know if the closure of the collection of complex Wiener roots contains a set of positive measure.

\medbreak

We have shown that the minimum real part of a Wiener root of a graph of order $n$ is precisely $-\left( \binom{n}{2} - 1 \right)$.  Currently the best result we have for the maximum positive real part, and the maximum (and minimum) imaginary parts are growth rates of $\Omega (\sqrt[3]{n})$ and $\Omega (n)$, respectively.

\begin{problem}
What connected graph of order $n$ has a Wiener root of (i) largest real part? and (ii) of largest imaginary part? What are the rates of growth (as a function of $n$) of these parameters?
\end{problem}

We have already seen that the roots of Wiener polynomials of almost all graphs are all real, since almost all graphs have diameter $2$, and any graph of diameter $2$ has all real Wiener roots.  We have also shown that the roots of the Wiener polynomials of all trees of order $n\geq 15$ and diameter $3$ are real.  In fact, there are trees of arbitrarily large diameter for which the Wiener roots are all real (and rational).  We describe such a family of trees recursively. Suppose $T_0$ is a tree with all real Wiener roots; for example, $P_3$ has this property. Let $T_1$ be the tree obtained from $T_0$ by adding a leaf to each of its vertices. Then
\[
W(T_1;x) = x^2W(T_0;x)+ 2xW(T_0;x) +W(T_0;x)=(x+1)^2W(T_0;x)
\]
and $D(T_1)=D(T_0)+2$. If $T_k$ has been defined for some $k \ge 1$, then let $T_{k+1}$ be obtained from $T_k$ by joining a leaf to each of its vertices. Then each $T_k$ has only real roots and $D(T_k) = D(T_0) + 2k$. If $T_0$ is chosen in such a way that it has only rational roots, then $T_k$ in fact has only rational roots.

\begin{problem}
What graphs have the property that their Wiener roots are all real? What trees have this property?
\end{problem}




\providecommand{\bysame}{\leavevmode\hbox to3em{\hrulefill}\thinspace}
\providecommand{\MR}{\relax\ifhmode\unskip\space\fi MR }
\providecommand{\MRhref}[2]{%
  \href{http://www.ams.org/mathscinet-getitem?mr=#1}{#2}
}
\providecommand{\href}[2]{#2}

\end{document}